\documentclass[11pt,final]{article}
\usepackage{amsmath}
\usepackage{amssymb}
\usepackage{amscd,todonotes}
\usepackage{latexsym,microtype}
\usepackage{theorem}
\usepackage{epsfig,euscript}
\usepackage{amsfonts,nicefrac}
\usepackage{xypic,xspace,tikz}
\usepackage{bbm,ifpdf}
\usepackage[notref,notcite]{showkeys}
\usepackage{setspace,mathrsfs}
\usepackage [margin=3cm]{geometry}
\ifpdf
  \usepackage{pdfsync,hyperref}
\fi

\newtheorem{theorem}{Theorem}[section]
\newtheorem{lemma}[theorem]{Lemma}
\newtheorem{proposition}[theorem]{Proposition}
\newtheorem{corollary}[theorem]{Corollary}

\newtheorem{conjecture}[theorem]{Conjecture}
\numberwithin{equation}{section}

{
\theorembodyfont{\rmfamily}
\theoremstyle{plain}

\newtheorem{example}[theorem]{Example}

\newtheorem{remark}[theorem]{Remark}
}

\newcommand{\qed}{\hfill \mbox{$\Box$}\medskip\newline}
\newenvironment{proof}{\noindent {\bf Proof:}}{\qed \par}
\newenvironment{proofweights}{\noindent {\bf Proof of Theorem \ref{weights}:}}{\qed \par}
\newenvironment{proofmain}{\noindent {\bf Proof of Theorem \ref{main}:}}{\qed \par}
\newenvironment{prooflss}{\noindent {\bf Proof of Proposition \ref{leaf semisimple}:}}{\qed \par}

\newcommand{\sgn}{\sigma}
\newcommand{\triv}{\operatorname{triv}}
\newcommand{\mg}{\mathfrak{g}}
\newcommand{\mh}{\mathfrak{h}}
\newcommand{\Irrep}{\operatorname{Irrep}}
\newcommand{\Spec}{\operatorname{Spec}}

\newcommand{\codim}{\operatorname{codim}}
\renewcommand{\dim}{\operatorname{dim}}

\newcommand{\End}{\operatorname{End}}
\newcommand{\Eu}{\operatorname{Eu}}
\newcommand{\rk}{\operatorname{rk}}
\newcommand{\crk}{\operatorname{crk}}
\newcommand{\Hom}{\operatorname{Hom}}
\newcommand{\Ext}{\operatorname{Ext}}

\newcommand{\Id}{\operatorname{Id}}

\newcommand{\bZ}{\mathbf{Z}}

\newcommand{\bC}{\mathbf{C}}
\newcommand{\cs}{\bC^\times}
\newcommand{\bD}{\mathbf{D}}

\newcommand{\la}{\lambda}

\newcommand{\into}{\hookrightarrow}
\newcommand{\surj}{\twoheadrightarrow}

\newcommand{\reg}{\operatorname{reg}}
\renewcommand{\H}{\operatorname{H}}
\newcommand{\bH}{\operatorname{\mathbf{H}}}
\newcommand{\IH}{\operatorname{IH}}

\newcommand{\IC}{\operatorname{IC}}
\newcommand{\hpz}{\operatorname{HP_0}}
\newcommand{\hp}{\operatorname{HP}}

\newcommand{\cA}{\mathcal{A}}

\newcommand{\cO}{\mathcal{O}}
\newcommand{\caD}{\mathcal{D}}

\newcommand{\tX}{\tilde X}
\newcommand{\rs}{\rho_*\Omega_{\tilde X}}

\newcommand{\fX}{\mathfrak{X}}
\newcommand{\tfX}{\tilde\fX}
\newcommand{\fT}{\mathfrak{T}}
\newcommand{\fM}{\mathfrak{M}}
\newcommand{\XA}{X(\cA)}
\newcommand{\tXA}{\tX(\tcA)}
\newcommand{\tcA}{\tilde{\cA}}
\renewcommand{\reg}{{\mathrm{reg}}}

\newcommand{\becircled}{\mathaccent "7017}

\newcommand{\C}{\bC}
\newcommand{\cB}{\mathcal{B}}

\newcommand{\excise}[1]{}
\renewcommand{\and}{\qquad\text{and}\qquad}

\newcommand{\nicktodo}{\todo[inline,color=green!20]}

\begin{document}
\spacing{1.2}

\noindent {\Large \bf 
Poisson-de Rham homology of hypertoric varieties\\ and nilpotent cones}\\

\noindent
{\bf Nicholas Proudfoot}\footnote{Supported by NSF grant DMS-0950383.}\\
Department of Mathematics, University of Oregon,
Eugene, OR 97403\\

\noindent
{\bf Travis Schedler}\\
Department of Mathematics, University of Texas,
Austin, TX 78712\\

% \subjclass[2010]{14E15, 14F10, 37J05} 
% \keywords{Poisson homology, de
% Rham cohomology, Poisson varieties, symplectic resolutions,
% hypertoric varieties, nilpotent cone, D-modules, Tutte polynomial,
% Kostka polynomial, Hamiltonian flow}

{\small
\begin{quote}
\noindent {\em Abstract.}
We prove a conjecture of Etingof and the second author for hypertoric
varieties, that the Poisson-de Rham homology of a unimodular
hypertoric cone is isomorphic to the de Rham cohomology of its
hypertoric resolution.  More generally, we prove that this conjecture
holds for an arbitrary conical variety admitting a symplectic
resolution
% whose symplectic leaves have conical slices 
if and only if it holds in degree zero for all normal slices to
symplectic leaves.
% We also prove a more general conjecture that, in this
% case, the D-module of Etingof and the second author is isomorphic to
% the pushforward of the canonical D-module on the resolution.

The Poisson-de Rham homology of a Poisson cone inherits a second
grading.
% In the preceding case, we prove a general result that computes this using
% the weight grading on the slices to the leaves.  
In the hypertoric case, we compute the resulting 2-variable Poisson-de
Rham-Poincar\'e polynomial, and prove that it is equal to a
specialization of an enrichment of the Tutte polynomial of a matroid
that was introduced by Denham \cite{Denham}.
% that encodes the dimensions of the eigenspaces of the combinatorial
% Laplacian of a matroid.  which is closely related to the Tutte
% polynomial.
We also compute this polynomial for S3-varieties of type A in terms
of Kostka polynomials, modulo a previous conjecture of the first
author, and we give a conjectural answer for nilpotent cones in
arbitrary type, which we prove in rank less than or equal to 2.
\end{quote}
}

\section{Introduction}
Let $X$ be a Poisson variety over $\C$.  Etingof and the second author
\cite{ESL} define a right D-module $M(X)$, and define the {\bf
  Poisson-de Rham homology group} $\hp_k(X)$ to be the cohomology in
degree $-k$ of the derived pushforward of $M(X)$ to a point.  If $X$
is affine, then $\hp_0(X)$ 
 coincides with the zeroth Poisson homology
of $\C[X]$, but $\hp_*(X)$ does not directly relate to higher Poisson
homology.  If $X$ is smooth and symplectic, then $M(X)$ is naturally
isomorphic to the right D-module $\Omega_X$ of volume forms on $X$,
and therefore we have an isomorphism $\hp_k(X) \cong \H^{\dim X -
  k}(X; \C)$.
% by the de Rham theorem.
The next natural case to consider is when $X$ is singular but admits a
conical symplectic resolution\footnote{A precise definition of a
  conical symplectic resolution is given at the beginning of Section
  \ref{sec:rigidity}.}  $\rho:\tX\to X$; examples include hypertoric
varieties, symmetric schemes of Kleinian singularities (more
generally, Nakajima quiver varieties), nilpotent cones (more
generally, S3-varieties), and certain slices to Schubert varieties in
the affine Grassmannian \cite[\S 2]{BLPWquant}.  In this case, Etingof
and the second author \cite[1.3.1]{ES11} conjecture that $M(X)$ is
(noncanonically) isomorphic to $\rs$, and therefore that $\hp_k(X)
\cong \H^{\dim X - k}(\tX; \C)$.

In this paper, we prove this conjecture for hypertoric varieties.
More generally, we show that if the vector space isomorphism holds
when $k=0$ not just for $X$, but also for all normal slices to
symplectic leaves of $X$, then the D-module isomorphism $M(X)\cong
\rs$ holds, as well (Theorem \ref{main}).  These vector space
isomorphisms have already been established for hypertoric varieties by
the first author \cite[3.2]{Pr12}, therefore Theorem \ref{main}
applies.

Part of the structure of a conical symplectic resolution is an action
of $\cs$ on $X$ with respect to which the Poisson bracket is
homogeneous.  The right D-module $M(X)$ is weakly $\cs$-equivariant,
and this induces a second grading on $\hp_*(X)$, which we call the
{\bf weight grading}.  We prove a general result (Theorem
\ref{weights} and its corollaries) that computes $M(X)$, with its
weight grading, in terms of the degree zero Poisson homology of the
slices.  Let $P_X(x,y)$ be the Poincar\'e polynomial of $\hp_*(X)$,
where $x$ encodes homological degree and $y$ encodes weight.  When $X$
is a hypertoric variety, we show that $P_X(x,y)$ is equal to a
specialization of a polynomial studied by Denham \cite{Denham} that
encodes the dimensions of the eigenspaces of the combinatorial
Laplacian of a matroid (Theorem \ref{laplacian}), which is closely
related to the Tutte polynomial of the associated hyperplane
arrangement. 
% We also give a formula using Tutte polynomials (see \eqref{e:tutte}).  
When $X$ is an S3-variety of type A, we similarly compute $P_X(x,y)$
in terms of Kostka polynomials (Proposition \ref{s3-kostka}), modulo a
conjecture that appears in \cite[3.4]{Pr12}.  Finally, we give a
conjectural description of $P_X(x,y)$ where $X$ is the nilpotent cone
in arbitrary type (Conjecture \ref{arbitrary type}), and prove it in
certain cases.

\vspace{\baselineskip}
\noindent
{\bf Acknowledgments:}
The first author would like to thank G.~Denham, C.~Mautner, and
V.~Ostrik for their help with this project, and in particular we thank
Mautner for help with Lemmas \ref{perversity} and \ref{l:ext-ss'}.
The second author would like to thank P.~Etingof for useful
discussions, particularly about Proposition \ref{semisimp} and
Conjecture \ref{arbitrary type}, and for helpful comments on an
earlier version. We are grateful to G.~Lusztig for suggesting the
formula of Conjecture \ref{arbitrary type}, and for introducing us to
the material in Remark \ref{BL}.

\section{The twistor family}
Let $X$ be a normal, irreducible, affine Poisson variety of finite
type over $\bC$.  Let $\rho:\tX\to X$ be a
% conical 
projective symplectic resolution, equipped with a particular choice of
ample line bundle on $\tX$.  Kaledin extends $\rho$ to a projective
map $\rho:\tfX\to\fX$ of schemes over the formal disk $\Delta := \Spec
\bC[[t]]$, where over the closed point $0\in\Delta$ we have
$$\tfX_0\cong \tX\and\fX\cong X.$$
Furthermore, he shows that $\fX$ is normal and flat over $\Delta$, and
that over the generic point, $\rho$ restricts to an isomorphism of
smooth, affine, symplectic varieties \cite[2.2 and 2.5]{KalDEQ}.  This
family of maps over $\Delta$ is called the {\bf twistor family}.

Let $M := M(X)$, and let $T := \rs$ be the derived pushforward to $X$ of the
right D-module of volume forms on $\tX$.  By \cite[2.11]{KalPPV},
$\rho$ is semismall, hence $T$ is also a right D-module (that is, the
homology of $T$ is concentrated in degree zero).  These extend
naturally to right D-modules $\fM := M(\fX)$ and $\fT :=
\rho_*\Omega_{\tfX}$ on $\fX$.  Let $\fM_0$ be the right D-module on
$\fX$ obtained by killing $\bC[[t]]$-torsion in $\fM$, and let $M_0$
be the restriction of $\fM_0$ to $X$. (In Theorem \ref{main}, we will
show that, under suitable hypotheses, $M_0$ is isomorphic to $M$.
However, we {\em a priori} know only that $M_0$ is a quotient of $M$.)

We would like to perform the same construction on $\fT$ and $T$, but
it is unnecessary: if we forget the Poisson structure, the family
$\tfX$ over $\Delta$ is locally trivial (that is, $\tfX$ admits an open
cover by trivial families over $\Delta$) \cite[17]{Namiflop},
thus $\fT$ has no $\bC[[t]]$-torsion.  Since $M(X)$ equals the
canonical D-module of volume forms when $X$ is smooth and symplectic
\cite[2.6]{ESL},
the right D-modules $\fM_0$, $\fM$, and $\fT$ are all isomorphic at
the generic fiber.

\begin{proposition}\label{semisimp}
The semisimplification of $M_0$ is (noncanonically) isomorphic to $T$.
\end{proposition}

\begin{proof}
  Since $\fM_0$ and $\fT$ are isomorphic at the generic fiber, the
  semisimplifications of $M_0$ and $T$ must be isomorphic (as they
  have the same class in the Grothendieck group of holonomic
  D-modules on $X$).  But $T$ is semisimple by the decomposition
  theorem \cite[6.2.5]{BBD},
% (see \cite[8.9.3]{CG97} and \cite[1.9]{Kal09}),
  so it must be isomorphic to the semisimplification of $M_0$.
\end{proof}

\section{Rigidity}\label{sec:rigidity}
We now add the hypothesis that $\rho:\tX\to X$ is {\bf
  conical}, which means the following:
\begin{itemize}
\item $X$ and $\tX$ are both equipped with actions of the
  multiplicative group $\cs$, and the map $\rho$ is equivariant.
\item The action of $\cs$ induces a non-negative grading on
  $\mathbb{C}[X]$, with only the constant functions in degree zero.
\item The Poisson bracket on $X$ (equivalently the symplectic form on
  $\tX$) is homogeneous for the action of $\cs$.
\end{itemize}
Our aim in this section is to prove that $\Ext^1(T,T) = 0$, and
therefore that $M_0$ is in fact isomorphic to $T$.  We accomplish this
in two steps, first showing that all summands supported on a single
leaf have no self-extensions, and then showing that there can be no
extensions between summands of $T$ supported on different leaves.

For the first step, we prove more generally that all topological local
systems on a leaf are semisimple.  We use the term {\bf local system} to
mean an $\mathcal{O}$-coherent right $\caD$-module (equivalently, a vector bundle with a flat connection) 
on a locally closed smooth subvariety.  
We use the term {\bf topological local system} to mean a representation of
the fundamental group of such a subvariety. By the Riemann-Hilbert
correspondence, the latter are equivalent to the former when we
require that the connection has regular singularities.  All of the
local systems we consider will have regular singularities.

\begin{proposition}\label{leaf semisimple}
  All finite-rank topological local systems on a leaf
  $S\subset X$ are semisimple.
\end{proposition}

\begin{remark}
  Proposition \ref{leaf semisimple} does not require that $X$ admit a
  symplectic resolution, but only that it be conical and be a
  symplectic variety in the sense of Beauville \cite{Beau}, which
  means that the Poisson bracket on the regular locus of $X$ is
  nondegenerate and the inverse meromorphic symplectic form extends to
  a (possibly degenerate) 2-form on some (equivalently every)
  resolution of $X$.  Such varieties include, for example, quotients
  of symplectic varieties by finite groups acting symplectically
  \cite[2.4]{Beau}, which often do not admit symplectic resolutions
  (see, e.g., \cite{BS-nesr}).
\end{remark}

\begin{prooflss}
  Let $Y$ be the normalization of the closure of $S$ in $X$.  Then $Y$
  is a symplectic variety in the sense of Beauville
  \cite[2.5]{KalPPV}, and the conical action on $X$ induces a conical
  action on $Y$.  The regular locus $Y_\reg$ is birational to $S$ and
  isomorphic away from a subvariety of codimension 2; in particular,
  the fundamental groups of $S$ and $Y_\reg$ are isomorphic.  Thus it
  is sufficient to prove that every finite-rank topological local
  system on $Y_\reg$ is trivial.

  Since $Y_\reg$ is a quasiprojective variety, $\pi_1(Y_\reg)$ is
  finitely generated (this follows, for example, from the finite
  triangulability of \cite{Loj-tsas}).  In the situation at hand,
  Namikawa has proved that the profinite completion $\hat
  \pi_1(Y_\reg)$ is finite \cite{Nam-Fgss}.  By a theorem of
  Grothendieck \cite{Gro-rlcpgd}, the map $\pi_1(Y_\reg) \to \hat
  \pi_1(Y_\reg)$ induces an equivalence of categories of
  finite-dimensional representations, hence the category of
  finite-dimensional representations of $\pi_1(Y_\reg)$ (equivalently,
  the category of finite-rank topological local systems on $Y_\reg$)
  is semisimple.
\end{prooflss}

Let $S\subset X$ be a symplectic leaf, and let $i:\bar S\smallsetminus
S\to X$ be the inclusion of the boundary of $S$.  Let $K_S :=
H^{\codim S}\rho_*\Omega_{\rho^{-1}(S)}$, which is a local system on
$S$ with regular singularities.  Since the resolution $\rho$ is
semismall \cite[2.11]{KalPPV}, the
%Beilinson-Bernstein-Deligne
decomposition theorem \cite[6.2.5]{BBD} 
% (see also \cite[8.9.14(b)]{CG97}) 
yields
\begin{equation}\label{e:t-dec}
T\;\cong\; \displaystyle\bigoplus_S \IC(S; K_S).
\end{equation}
%and that the $K_S$ are semisimple (note that, for us, this last fact
%also follows from Proposition \ref{leaf semisimple}).
%By Proposition \ref{leaf semisimple}, the $K_S$ are all semisimple.
By Proposition \ref{leaf semisimple}, we conclude that 
$\Ext^1(K_S, K_S) = 0$, and therefore that we have $\Ext^1(\IC(S;
K_S), \IC(S; K_S)) = 0$ for all $S$.

It remains to show that there are no extensions between summands on
different leaves.  We do this using the following two
lemmas.\footnote{The authors thank Carl Mautner for explaining the
  following two lemmas and their proofs.}
\begin{lemma}\label{perversity}
  The complex $i_*i^*\IC(S; K_S)$ of right D-modules is concentrated
  in degrees $\leq -2$.
\end{lemma}

\begin{proof}
  It is a standard property of intermediate extensions of local
  systems that $i^* \IC(S; K_S)$ is concentrated in negative degrees.
  % Let $j: S \to X$ be the inclusion. Then $\IC(S; K_S) := j_{!*}
  % K_S$, and it is a
  % standard property that $i^* \IC(S;K_S) = \tau^{< 0} (i^* j_*
  % \IC(S;K_S))$, where $\tau^{< 0}$ is the truncation.  In
  % particular,
  % $i^* \IC(S;K_S)$ is concentrated in negative degrees.
  Since $i$ is a closed embedding, $i_*$ is exact, and thus $i_* i^*
  \IC(S; K_S)$ is concentrated in negative degrees.  Therefore we only
  have to show that $\H^{-1} i_* i^* \IC(S; K_S) = 0$.

  For a contradiction, let $S'$ be a maximal symplectic leaf in the
  closure of $S$ on which $\H^{-1} i_* i^* \IC(S;  K_S)$ is supported,
  and let $j_{S'}: S' \to X$ the inclusion.  Then $$\H^{-1} j_{S'}^*
  i_* i^* \IC(S; K_S) = \H^{-1} j_{S'}^* \IC(S;  K_S)$$ is a local
  system on $S'$.  By our assumption, the stalk of $\IC(S; K_S)$ at
  every point of $S'$ has nonzero cohomology in degree $-\dim S'-1$.
  
  Choose a point $x \in S'$.  The stalk $\IC(S; K_S)_x$ is a summand
  of $T_x$. But $\H^*(T_x)$ is the pushforward to a point of the
  restriction of $\Omega_{\tilde X}$ to the fiber $\rho^{-1}(x)$.
  This is the same for the formal neighborhood (or an analytic
  neighborhood) of $\rho^{-1}(x)$, thus we obtain the shifted
  topological cohomology $\H^{* + \dim X}(\rho^{-1}(x); \bC)$ of the
  fiber.  By \cite[1.9]{Kal09}, $\H^*(\rho^{-1}(x); \bC)$ is
  concentrated in even degrees, and hence the same is true for
  $\H^*(T_x)$. Since $\dim S'$ is even, this gives us a contradiction.
\end{proof}

\begin{lemma}\label{l:ext-ss'}
  Let $S \neq S'$ be symplectic leaves of $X$. Then
  $\Ext^1(\IC(S;K_S), \IC(S',K_{S'})) = 0$.
\end{lemma}
\begin{proof}
  Assume first that $S$ is not contained in the closure of $S'$.
  Thus, $S$ is disjoint from the closure of $S'$.  Let $j_S: S \to X$
  and $i_{S}: \bar S \setminus S \to X$ be the inclusions.  Then
  $j_S^* \IC(S', K_{S'}) = 0$.  We have the standard exact triangle
\[
\to (j_S)_! K_S \to \IC(S;  K_S) \to (i_S)_* i_S^* \IC(S;  K_S) \to .
\]
Apply $\Hom\left(-, \IC(S', K_{S'})\right)$, and we obtain in the long exact
sequence,
\begin{multline*}
  \to \Ext^1((i_S)_* i_S^* \IC(S;  K_S), \IC(S', K_{S'})) \to
  \Ext^1(\IC(S;  K_S), \IC(S', K_{S'})) \\ \to \Ext^1((j_S)_! K_S,
  \IC(S', K_{S'})) \to.
\end{multline*}
We want to show that the middle term is zero.  By adjunction, since
$j_S^* \IC(S', K_{S'}) = 0$, the last term is zero. It suffices
therefore to show that the first term is zero.  However, by Lemma
\ref{perversity}, $(i_S)_* i_S^* \IC(S;  K_S)$ has cohomology
concentrated in degrees $\leq -2$, whereas $\IC(S', K_{S'})$ is a
D-module (in degree zero).  Therefore, the first term is also
zero.\footnote{More generally, for any triangulated category with a
  $t$-structure, if $M$ is a complex whose cohomology is concentrated
  in negative degrees and $N$ is a complex whose cohomology is
  concentrated in nonnegative degrees, then $\Hom(M, N) = 0$.}

Next assume $S$ is contained in the closure of $S'$. Since $S \neq
S'$, $S'$ is not contained in the closure of $S$.  In this case,
applying Verdier duality,
$$\Ext^1(\IC(S;  K_S), \IC(S', K_{S'})) = \Ext^1(\bD \IC(S', K_{S'}), \bD 
\IC(S;  K_{S})).$$ But, since $\Omega_{\tilde X}$ is self-dual, so is
$T$, and hence $\bD \IC(S;  K_S) = \IC(S;  \bD K_S)$ is a summand of
$T$. Therefore, $\IC(S;  \bD K_S) \cong \IC(S;  K_S)$, and the same
holds for $S'$. Thus we again have $\Ext^1(\IC(S;  K_S), \IC(S', K_{S'}))
= 0$.
\end{proof}

Putting together Lemma \ref{l:ext-ss'} and Proposition \ref{leaf
  semisimple}, \eqref{e:t-dec} immediately implies:
\begin{proposition}\label{T is rigid}
  The D-module $T$ is rigid; that is, $\Ext^1(T,T) = 0$.
\end{proposition}

The following corollary is an immediate consequence of Propositions
\ref{semisimp} and \ref{T is rigid}.

\begin{corollary}\label{quotient}
  $M_0$ is isomorphic to $T$.
\end{corollary}

\section{The main theorem}
For each leaf $S$, choose a point $s\in S$, and let $X_S$ be a
formal slice to $S$ at $s$.  Then the base change of $\rho$ along
the inclusion of $X_S$ into $X$ induces a projective symplectic resolution
$\tX_S\to X_S$.  The following theorem asserts that, if
\cite[Conjecture 1.3.(a)]{ES11} holds for each $X_S$, then
\cite[Conjecture 1.3.(c)]{ES11} holds for $X$.

\begin{theorem}\label{main}
  Suppose that, for every leaf $S$ of $X$, $\dim\hpz(X_S) = \rk K_S$.
  Then $M \cong T$.
\end{theorem}
Note that, in the above theorem, the isomorphism $M \cong T$ is
\emph{not} canonical.  This can be corrected as follows.  Let $i: S
\to X$ be the inclusion, and let $L_S := \H^0(i^*M)$ be the local
system studied in \cite[\S 4.3]{ESL}.\footnote{It is not {\em a
    priori} clear that $L_S$ has regular singularities, though this
  will follow from Corollary \ref{c:m-dec}.}  The fiber $L_{S,s}$ of
$L_S$ at the point $s$ is canonically isomorphic to the vector space
$\hpz(X_S)$ \cite[4.10]{ESL}.

\begin{corollary}\label{c:m-dec}
  Under the hypothesis of Theorem \ref{main}, there is a canonical
  isomorphism $M \cong \bigoplus_S \IC(S; L_S)$, and a noncanonical
  isomorphism $L_S \cong K_S$ for each symplectic leaf $S$.
\end{corollary}
\begin{proof}
  As explained in \cite[\S 4.3]{ESL}, the right D-modules $\IC(S;
  L_S)$ are subquotients of $M$. Theorem \ref{main} and Equation
  \eqref{e:t-dec} together imply that $L_S$ is a subquotient of $K_S$
  for all $S$.  Since $\rk L_S = \dim\hpz(X_S) = \rk K_S$, $K_S$ and
  $L_S$ are in fact isomorphic.

  Consider the canonical adjunction morphism $M \to H^0(i_* L_S)$,
  which induces a surjection from $M$ to $\IC(S; L_S)$.  By
  Proposition \ref{T is rigid}, the map $M\to \bigoplus_S \IC(S; L_S)$
  is an isomorphism.
\end{proof}

\begin{proofmain}
  Let $N$ be the kernel of the surjection $M\to M_0\cong T$, so that
  we have a short exact sequence
$$0\to N\to M\to T\to 0$$ of right D-modules on $X$.
\excise{ Our first claim is that the support of $N$ is disjoint from
  the open leaf $\becircled{X}\subset X$.  Indeed, since the inclusion
  of $\becircled{X}$ into $X$ is an open Poisson embedding of a smooth
  symplectic variety, the restriction of $M$ to $\becircled{X}$ is
  isomorphic to the canonical sheaf.  Since $\rho$ is an isomorphism
  over $\becircled{X}$, the restriction of $T = \rs$ to
  $\becircled{X}$ is also the canonical sheaf.  It is easy to check
  that the map between them is an isomorphism.  }  Assume for the sake
of contradiction that the support of $N$ is nontrivial.  It is
necessarily a union of symplectic leaves; let $S$ be a maximal such
leaf.  Restrict to the formal neighborhood of the leaf $S$.  Then $N,
M$, and $T$ are local systems along $S$ (that is, upon restriction to
a contractible analytic open neighborhood $U$ of every point of $S$,
they become external tensor products of local systems on $U$ and
D-modules on the normal slice).  We can therefore make use of the
exact restriction functor for such D-modules to the slice $X_S$ at
$s$, given by $P \mapsto P_S := i_{X_S}^*P[-\dim S]$ (for $i_{X_S}$
the inclusion of $X_S$ into the formal neighborhood of $S$), and we
obtain the exact sequence
 $$0\to N_S\to M_S\to T_S\to 0.$$
 By functoriality and the definitions of $M$ and $T$, $M_S$ is
 isomorphic to $M(X_S)$, and $T_S$ is isomorphic to the derived
 pushforward of the canonical sheaf of $\tX_S$.

Let $\pi$ be the pushforward of $X_S$ to a point.  We
have $$\H^{-1}\pi_*(T_S) \cong \H^{\dim X_S - 1}\!\big(\rho^{-1}(s);
\bC\big) = 0$$ by \cite[2.12]{KalPPV}.  Also, $N_S$ is a
delta-function D-module at $s$, so $\pi_* N_S$ is concentrated in
degree zero.  Thus, we obtain a short exact sequence
$$0 \to \H^0\pi_*(N_S)\to \H^0\pi_*(M_S)\to \H^0\pi_*(T_S)\to 0.$$
We have $$\H^0\pi_*(T_S) \cong \H^{\dim X_S}(\rho^{-1}(s); \bC).$$ By
assumption, we also have $$\dim \H^0\pi_*(M_S) = \dim \hpz(X_S) = \rk
K_{S} = \dim \H^{\dim X_S}\!\big(\rho^{-1}(s); \bC\big).$$ Thus
$\H^0\pi_*(M_S)$ and $\H^0\pi_*(T_S)$ have the same dimension, and
therefore $\H^0\pi_*(N_S)=0$.  This means that $N_S = 0$, which is a
contradiction.
\end{proofmain}

Let $\pi$ be the map from $X$ to a point.  By definition, we have
$\hp_k(X) := \H^{-k} \pi_*(M)$.  By the de Rham theorem, we have an
isomorphism $\H^{-k} \pi_*(T) \cong \H^{\dim X - k}(\tX; \bC)$.  Thus,
as explained by Etingof and the second author \cite[1.3]{ES11},
Theorem \ref{main} implies that the Poisson-de Rham homology of $X$ is
isomorphic to the de Rham cohomology of $\tX$.

\begin{corollary}\label{conjecture b}
  Under the hypothesis of Theorem \ref{main}, we have a (noncanonical)
  isomorphism $\hp_k(X)\cong \H^{\dim X - k}(\tX; \bC)$ for all $k$.
\end{corollary}

\begin{corollary}\label{poincare}
  Any two conical projective symplectic resolutions of $X$ have the
  same Betti numbers.
  % Poincar\'e polynomial.
\end{corollary}

\begin{remark}
  In fact, it is possible to show that any two conical projective
  symplectic resolutions of $X$ have canonically isomorphic cohomology
  rings; this follows from \cite[25]{Namiflop}.
\end{remark}

\begin{example}\label{first hypertoric}
  Let $\cA$ be a coloop-free, unimodular, rational, central hyperplane
  arrangement, and let $\XA$ be the associated hypertoric variety
  \cite[\S 1]{PW07}.  Any simplification $\tcA$ of $\cA$ determines a
  conical projective symplectic resolution $\tXA$ of $\XA$.  The
  symplectic leaves of $\XA$ are indexed by coloop-free flats of
  $\cA$, and the slice to the leaf indexed by $F$ is isomorphic to a
  formal neighborhood of the cone point of $X(\cA_F)$, where $\cA_F$
  is the localization of $\cA$ at $F$ \cite[\S 2]{PW07}.  Hence the
  hypothesis of Theorem \ref{main} is that, for every coloop-free
  flat $F$,
$$\dim\hpz(X(\cA_F)) = \dim\H^{2\rk F}\!\!\big(\tX(\tcA_F); \bC\big).$$
This is proved in \cite[3.2]{Pr12}, hence Theorem \ref{main} holds for
hypertoric varieties.
\end{example}

\section{Weights}
We assume throughout this section that the hypothesis of Theorem
\ref{main} is satisfied. 

By homogeneity of the Poisson bracket, the vector space $$\hpz(X)
\cong \bC[X]\big{/}\{\bC[X], \bC[X]\}$$ inherits a grading from the
action of $\cs$.  Moreover, the D-module $M$ has a canonical weak
$\cs$-equivariant structure, thus $\hp_k(X) = \H^{-k} \pi_*(M)$ is
naturally graded for all $k$ (where $\pi$ is the map from $X$ to a
point).

% in addition, the local systems $\{K_S\}$ are all trivial.  This
% condition is satisfied by many interesting examples, including all
% hypertoric varieties \cite[5.2]{PW07} and quiver varieties
% \cite[14.3.2(1)]{Na}.  \nicktodo{We also need an additional
% assumption that the slices are cones, and that their conical
% structures are somehow compatible with the action on $X$.}

Let $n$ be the positive integer such that the Poisson bracket on $X$
has weight $-n$ (this weight must be negative since the bracket
vanishes along $\rho^{-1}(0)$ in the resolution $\tX$).  Suppose that,
for every symplectic leaf $S$, the normal slice $X_S$ admits a conical
$\cs$-action equipping the Poisson bracket on $X_S$ with the same
weight $-n$.  More generally, we can suppose $X_S$ to be equipped with
a vector field $\xi$ such that $L_\xi \pi_{X_S} = -n \pi_{X_S}$, where
$\pi_{X_S}$ is the Poisson bivector (that is, we only require an
infinitesimal action of $\cs$).  In fact, this is no additional
assumption: such a $\xi$ always exists by virtue of the
Darboux-Weinstein decomposition $\hat X_s \cong \hat S_s \hat \times
X_S$ \cite[2.3]{KalPPV}.  If $p: \hat X_s \to X_S$ is the projection,
we may take $\xi = p_* (\Eu_{\hat X_s}|_{\{0\} \times X_S})$, where
$\Eu_{\hat X_s}$ is the vector field for the $\cs$-action.
% , i.e.,
% $\Eu_{\hat X_s} = \xi + \eta + \eta'$ where $\eta$ is parallel to
% $\hat S_s$ (equivalently, $\eta = f \xi_g$ for $g \in \hat \cO_{S,s}$)
% and $\eta'$ vanishes at $\{0\} \hat \times X_S$.  
However, we impose no requirement that the vector field $\xi$ be
obtained in this way.

\excise{ , and that for every $S$
  there is a point $s \in S$ such that the formal neighborhood of $X$
  along the punctured line $\cs \cdot s$ decomposes as a product of
  $\cs$-equivariant Poisson schemes,
\begin{equation}\label{e:pr-ass}
\hat X_{\cs \cdot s} \cong \hat S_{\cs \cdot s} \hat \times X_S.
\end{equation}
We will explain below that this technical hypothesis is satisfied for
hypertoric varieties.
}

Let $\pi$ be the Poisson bivector on $X$, and let $\theta$ be any
vector field such that $L_\theta \pi = c \pi$ for some $c \in \bC$.
Then the bracket of $\theta$ with any Hamiltonian vector field is
again Hamiltonian, thus left multiplication by $\theta$ is an
endomorphism of the right D-module $M(X)$.  It is the zero
endomorphism if and only if $\theta$ is Hamiltonian (in which case
$c=0$).  Since $\IC(S; L_S)$ is (canonically) a quotient of $M(X)$, this
also induces an endomorphism of $\IC(S; L_S)$, and hence of $L_S$ and
its fiber $L_{S,s}$.
In the case of the Euler vector field $\Eu_X$,
which is induced by an honest action of $\cs$, this endomorphism must
be semisimple.  By the same construction, the vector field $\xi$ on
$X_S$ induces an endomorphism of $M(X_S)$, and therefore of the vector
space $\hpz(X_S)$.  In this case, since we do not assume that $\xi$
integrates to an honest action of $\cs$, we do not know {\em a priori}
that the endomorphism is semisimple.  The following result says that
it is, and that the induced gradings on $L_{S,s} \cong \hpz(X_S)$
agree up to a shift.

\begin{theorem}\label{weights}
  The endomorphism of $\hpz(X_S)$ induced by $\xi$ is semisimple, and
  the canonical vector space isomorphism $L_{S,s} \to \hpz(X_S)[n\dim
  S/2]$ respects the weight gradings.
%  We have an equality $P(L_S;t) = t^{-n\dim S /2} P(\hpz(X_S);t)$.
\end{theorem}

In many cases, including hypertoric and Nakajima
quiver varieties, the local systems $L_S \cong K_S$ are trivial.
This allows us to conclude the following two corollaries. 

\begin{corollary}\label{Mweights} If the local systems $\{K_S\}$ are trivial, then there is an isomorphism of
  weakly $\cs$-equivariant D-modules
$$M(X) \cong \sum_S \IC(S; L_S) \cong \sum_S \IC(S; \Omega_S) \otimes \hpz(X_S)[n\dim S/2],$$
where the grading on $\hpz(X_S)$ is induced by the
(possibly infinitesimal) action of $\cs$ on $X_S$.
\end{corollary}

Define $P_X(x,y)$ to be the Poincar\'e polynomial of $\hp_*(X)$, where
$x$ records homological degree and $y$ records weights for the $\cs$-action.  
Note that weights can be both positive and negative, so
$P_X(x,y)$ is a polynomial in $x$, $y$, and $y^{-1}$.  For each leaf
$S$, let $Q_{\bar S}(x)$ be the intersection cohomology
Poincar\'e polynomial of $\bar S$, that is, $Q_{\bar S}(x) := \sum \dim\IH^k(\bar S; \C)\, x^k$.
\begin{corollary}\label{pp}
If the local systems $\{K_S\}$ are trivial, then
$$P_X(x,y) \; =\; \sum_S \; x^{\dim S}\; y^{-n\dim
    S/2}\; Q_{\bar S}(x^{-1})\; P_{X_S}(0, y).$$
\end{corollary}

\begin{proof}
Let $\pi$ be the map from $X$ to a point.  Then the corollary follows from
Corollary \ref{Mweights} and the fact that $\H^{-k}\!\left(\pi_* \IC(S; \Omega_S)\right) \cong \IH^{\dim S - k}(\bar S; \bC)$. 
\end{proof}

\vspace{-\baselineskip}
\begin{remark}\label{duality}
  The first author has conjectured that, if $n=2$ and $X^!$ is {\bf
  symplectic dual} to $X$ in the sense of \cite[10.15]{BLPWgco},
  then $P_X(0, y) = Q_{X^!}(y)$ \cite[3.4]{Pr12}.  Thus, if each $K_S$ is trivial, each slice
  $X_S$ has a symplectic dual, and the aforementioned conjecture holds, then
  Corollary \ref{pp} allows us to express $P_X(x,y)$ entirely in terms
  of intersection cohomology Poincar\'e polynomials.\footnote{As explained in \cite{Pr12},
  the full version of the conjecture \cite[3.4]{Pr12} applied to a slice $X_S$ with a symplectic
  dual would imply that $\dim\hpz(X_S) = \rk K_S$, thus the hypothesis of Theorem \ref{main}
 would be satisfied.} 

In the next two sections, we will apply this result to compute
$P_X(x,y)$ when $X$ is a unimodular hypertoric cone or a type A
S3-variety.  In both of these cases, the leaf closures, the slices,
and their symplectic duals are varieties of the same type \cite[10.4,
10.8, 10.16, 10.18, and 10.19]{BLPWgco}; the local systems are
trivial; and we know how to compute their intersection cohomology
Poincar\'e polynomials.  In the former case we obtain $h$-polynomials
of the broken circuit complexes of matroids, and in the latter case we
obtain Kostka polynomials.  The conjecture about symplectic duals is
proved for hypertoric varieties but not for S3-varieties, thus the
computations in Section \ref{sec:kostka} are conditional on this
unproved statement.  
  % Note that the full conjecture in \cite[3.4]{Pr12} implies the
  % conjecture \cite[1.3.(a)]{ES11} (as pointed out there) for
  % varieties with a symplectic dual, so by our main result, when all
  % the slices $X_S$ have symplectic duals, it would imply
  % \cite[1.3.(c)]{ES11}.
\end{remark}

To prove Theorem \ref{weights}, we need the following result.  Let
$U_s$ be a contractible open neighborhood of $s$ in the analytic
topology.

\begin{proposition}\label{ham}
Every analytic Poisson vector field on $U_s$ is Hamiltonian, and so is any algebraic Poisson vector field on the formal completion $\hat X_s$.\end{proposition}

\begin{proof}
  By normality of $X$, it is enough to prove each statement on the
  regular locus.  On a smooth symplectic manifold, Poisson vector
  fields correspond to closed one-forms and Hamiltonian vector fields
  correspond to exact one-forms. Thus, we need to show that the global
  sections of the de Rham complex has vanishing first cohomology on
  the regular locus.  For the analytic statement, it suffices to show
  that the topological cohomology of $U_s^{\reg}$ vanishes, since in
  this case every closed one-form is the differential of a smooth
  function, and if the one-form is analytic, the same must be true of
  the function.

We begin by observing that $U_s^{\reg} \cong \rho^{-1}(U_s^{\reg})$.  By \cite[2.12]{KalPPV},
  $\H^1(\rho^{-1}(U_s); \C)=0$, so we need to show that passing to the preimage of the regular
  locus does not introduce any cohomology in degree 1.
Since $U_s \setminus U_s^{\reg}$ has complex codimension at least one, hence
  real codimension at least two, every loop in $U_s$ can be homotoped
  to $U_s^{\reg}$, so the map $\pi_1(U_s^{\reg}) \to \pi_1(U_s)$ is
  surjective.  Consider the stratification
 \[
\rho^{-1}(U_s \setminus U_s^{\reg}) = \bigsqcup_{S'} \rho^{-1}(U_s \cap S'),
\]
where $S'$ ranges over all symplectic leaves of $X$ whose closure
contains $S$ other than the open leaf.  Suppose $S'$ is such a leaf.
By the semismallness property \cite[2.11]{KalPPV}, the codimension of
$\rho^{-1}(U_s \cap S')$ is at least half the codimension of $S'$. If
the codimension of $S'$ is at least four, then $\rho^{-1}(U_s \cap
S')$ therefore has codimension at least two, hence real codimension
four.  The fundamental group of a smooth manifold is unchanged by
removing a locus of real codimension greater than two (since
homotopies of loops can be pushed off this locus).  Therefore the
fundamental group of $U_s$ is unchanged by removing the union of
$\rho^{-1}(U_s \cap S')$ over all leaves $S'$ of codimension at least
four.  Next, if $S'$ has codimension two, then the singularity at $s'
\in S'$ is of Kleinian type, and hence in a small enough neighborhood
$U_{s'}$ of $s'$, the fundamental group $\pi_1(\rho^{-1}(U_{s'}))$ is
a finite subgroup of $\operatorname{SL}_2(\bC)$.  Therefore, the
kernel of the map $\pi_1(\rho^{-1}(U_s \setminus S')) \to
\pi_1(\rho^{-1}(U_s))$ is generated by this finite subgroup of
$\operatorname{SL}_2(\bC)$.  We conclude that the
surjection $\pi_1(\rho^{-1}(U_s^{\reg})) \to \pi_1(\rho^{-1}(U_s))$ is
generated by elements of finite order, and hence this surjection
descends to an isomorphism on homology $$0 = \H_1(\rho^{-1}(U_s^{\reg}),\bC)
\overset{\cong}{\longrightarrow} \H_1(\rho^{-1}(U_s), \bC).$$  Dualizing, we obtain the desired result.

For the statement about the formal completion, we follow \cite[\S 4.4]{ES-dmlv}. Let $V := U \setminus
U_s^{\reg}$ be the singular locus. By Hartshorne's theorem
\cite{Har-adrc,Har-drcav}, the de Rham hypercohomology of the formal
completion $\hat {\tX}_{\rho^{-1}(s)}$ equals the topological
cohomology of the fiber $\rho^{-1}(s)$, which also equals the
topological cohomology of $\rho^{-1}(U_s)$.  Then, as in
\cite[(4.40)]{ES-dmlv}, the Mayer-Vietoris sequences for the triples
$(\tX, \tX \setminus \rho^{-1}(V), \hat {\tX}_{\rho^{-1}(s)})$ and
$(\tX, \tX \setminus \rho^{-1}(V), \rho^{-1}(U_s))$ are
isomorphic. Since the intersections of the second two open subsets of
$\tX$ are $\rho^{-1}(\hat X_s^{\reg}) \cong \hat X_s^{\reg}$ and
$\rho^{-1}(U_s^{\reg}) \cong U_s^{\reg}$, respectively, we may take hypercohomology of the de Rham complex to conclude
that  $\bH^1_{DR}(\hat X_s^{\reg}) \cong \bH^1_{DR}(U_s^\reg)$.  The
latter, by Grothendieck's theorem, is equal to the topological
cohomology, which
we showed is zero.  

To conclude, we need to compare the hypercohomology of the de Rham
complex with the cohomology of global sections of the de Rham complex.
The spectral sequence computing hypercohomology degenerates in degree
one on the second page, yielding an isomorphism $\bH^1_{DR}(\hat
X_s^\reg) \cong H^1 (\Gamma(\Omega^\bullet_{\hat X_s^\reg})) \oplus
(R^1 \Gamma)(\cO_{\hat X_s^\reg})$.  Hence both summands are zero.
\end{proof}

\begin{proofweights}
  \excise{ \nicktodo{If $X_S$ is the completion of a variety with an
      honest $\cs$-action, then Namikawa's result can be use to show
      that the profinite completion of the smooth locus of $X_S$ is
      finite, which should imply the desired result.  } 
  Let $G(X)$ denote the set of vector fields $\theta$ on $X$ such that
  $L_\theta \pi = c \pi$ for some $c \in \bC$, and let $\nu: G(X) \to
  \bC$ be the map $\theta \mapsto c$.  Make the same definitions for
  $U_s$ and $\hat X_s$.  Then the maps $G(U_s) \to \End(M(U_s))$ and
  $G(\hat X_s) \to \End(M(\hat X_s))$ factor through $\nu$ by the
  above.  } Passing to a formal neighborhood of $s$, we obtain the
Darboux-Weinstein decomposition $\hat X_s \cong \hat S_s \hat \times
X_S$.  Then we can view $\xi$ as a vector field on $\hat X_s$ parallel
to the $X_S$ factor everywhere, and set $\xi' := \Eu_{\hat S_s} \!\!+
\,\xi$.  Letting $\pi$ be the Poisson bivector on $X$ and hence on
$\hat X_s$, we have $L_{\xi'} \pi = -n \pi = L_{\Eu} \pi$.  This
implies that $\xi'-\Eu$ is Poisson, and therefore Hamiltonian by
Proposition \ref{ham}.  Thus the endomorphisms of $M(\hat X_s)\cong
M(\hat S_s)\boxtimes M(X_S)$ induced by $\Eu$ and $\xi'$ are equal.
  
The endomorphism induced by $\Eu$ is responsible for the grading on $L_{S,s}$, and the endomorphism of $M(X_S)$
induced by $\xi$ is responsible for the grading on $\hpz(X_S)$.  To prove the theorem,
we need to show that the endomorphism of $M(\hat S_s)$ induced by $\Eu_{\hat S_s}$ is multiplication by
$-n\dim S/2$.
  To see this last fact, note that $M(\hat S_s) \cong \Omega_{\hat
    S_s}$ via the map that
  sends the canonical generator to $\omega^{\dim S/2}$, where $\omega$ is the symplectic form on $\hat S_s$.
  Since the Lie derivative map from vector fields to differential operators is an antihomomorphism,
  we have 
  $$\Eu_{\hat S_s}\cdot\; \omega^{\dim S/2} = -L_{\Eu_{\hat S_s}}\omega^{\dim S/2}
  = -(n\dim S/2)\,\omega^{\dim S/2},$$
by our assumption that the Poisson bracket on $X$, and hence on $S$, has
weight $-n$. 
 \end{proofweights}

\section{The hypertoric case}
In this section we compute the polynomial $P_{\XA}(x,y)$ for a
coloop-free, unimodular, rational, central hyperplane arrangement
$\cA$ with $\ell$ hyperplanes.  We use the action of $\cs$ described
in \cite[\S 2]{Pr12}, for which the symplectic form on the resolution
has weight $n=2$.  
% The symplectic leaves are indexed by coloop-free flats $F \subseteq
% \XA$.

Denham \cite[\S 3]{Denham} defines a polynomial
$\Phi_{\!\cA}(x,y,b_1,\ldots,b_\ell)$ whose coefficients are the
dimensions of certain eigenspaces (determined by the $b$ exponents) of
``combinatorial Laplacian" operators on certain vector spaces
(determined by the $x$ and $y$ exponents).  We will identify all of
the $b$ variables to obtain a 3-variable polynomial
$\Phi_{\!\cA}(x,y,b)$.  This is an enrichment of the Tutte polynomial in the sense
that $\Phi_{\!\cA}(x-1,y-1,1) = T_{\cA}(x,y)$ \cite[23(2)]{Denham}.

\begin{theorem}\label{laplacian} $\;P_{\XA}(x,y) \; = \;
  y^{-2\rk\cA}\;\Phi_{\!\cA}(x^2+1, y^{-2}+1, y^2).$
\end{theorem}

\begin{proof}
  As stated in Example \ref{first hypertoric}, the symplectic leaves
  of $\XA$ are indexed by coloop-free flats of $\cA$, and the leaf
  indexed by $F$ has a formal slice that is isomorphic to a formal
  neighborhood of the cone point of $X(\cA_F)$.  Furthermore, the
  closure of the leaf is isomorphic to $X(\cA^F)$, where $\cA^F$ is
  the restriction of $\cA$ to $F$ \cite[\S 2]{PW07}.

  Let $T_\cA(x,y)$ be the Tutte polynomial of $\cA$.  By \cite[4.3 and
  5.5]{PW07}, we have
$$Q_{X(\cA)}(x) = h_\cA^{\operatorname{br}}(x^2) = x^{2\rk\cA}\;T_\cA(x^{-2}, 0).$$
Applying this to the restricted arrangement $\cA^F$, we
obtain $$Q_{X(\cA^F)}(x) = x^{2\crk F}T_{\cA^F}(x^{-2}, 0),$$ where
$\crk F = \rk\cA - \rk F$.  By \cite[3.1]{Pr12}, we have
$$P_{X(\cA)}(0,y) = Q_{X(\cA^\vee)}(y) = y^{2\rk\cA^\vee}\;T_{\cA^\vee}(y^{-2}, 0)
= y^{2|\cA| - 2\rk\cA}\;T_\cA(0, y^{-2}),$$ where $\cA^\vee$ is the
Gale dual of $\cA$.  Applying this to the localized arrangement
$\cA_F$, we obtain
$$P_{X(\cA_F)}(0, y) = y^{2|F| - 2\rk F}\;T_{\cA_F}(0, y^{-2}).$$
Applying Corollary \ref{pp}, we have
\begin{eqnarray*}\label{tutte-eq}\notag
  P_{\XA}(x,y) &=& \sum_F\; x^{2\crk F} y^{-2\crk F} x^{-2\crk F}\;T_{\cA^F}(x^2, 0)\;
  y^{2|F| - 2\rk F}\; T_{\cA_F}(0, y^{-2})\\
  &=& y^{-2\rk \cA} \;\sum_F\; y^{2|F|}\;T_{\cA^F}(x^2, 0)\; T_{\cA_F}(0, y^{-2}). \label{e:tutte}
\end{eqnarray*}
Let $\chi_\cA(x) = (-1)^{\rk\cA}\;T_\cA(1-x, 0)$ be the characteristic
polynomial of $\cA$.  By the first equation in \cite[\S 3.1]{Denham},
we have
\begin{eqnarray*}
  \Phi_{\!\cA}(x,y,b) 
  &=& \sum_F\; (-1)^{\rk \cA - |F|} \;\chi_{\cA^F}(-x) \;\chi_{(\cA_F)^\vee}(-y)\; b^{|F|}\\
  &=& \sum_F\; (-1)^{\rk \cA - |F|}\; (-1)^{\rk \cA^F} (-1)^{\rk (\cA_F)^\vee} 
  T_{\cA^F} (x-1, 0)\; T_{(\cA_F)^\vee}(y-1, 0)\; b^{|F|}\\
  &=& \sum_F\; T_{\cA^F} (x-1, 0)\; T_{(\cA_F)^\vee}(y-1, 0)\; b^{|F|}\\
  &=& \sum_F\; T_{\cA^F} (x-1, 0)\; T_{\cA_F}(0, y-1)\; b^{|F|}.
\end{eqnarray*}
Thus $$y^{-2\rk\cA}\;\Phi_{\!\cA}(x^2+1, y^{-2}+1, y^2) =
y^{-2\rk\cA}\; \sum_F\; T_{\cA^F} (x^2, 0)\; T_{\cA_F}(0, y^{-2})\;
y^{2|F|} = P_{\XA}(x,y),$$ and the theorem is proved.
\end{proof}

\begin{remark}
  Specializing at $y=1$, we obtain the equation $$P_{\XA}(x,1) =
  \Phi_{\!\cA}(x^2+1, 2, 1) = T_{\cA}(x^2,1) = x^{2\rk\cA}
  h_{\cA}(x^{-2}),$$ matching the known formula for the Betti numbers
  of a conical symplectic resolution of $\XA$ given in \cite[1.2]{HS}
  and \cite[3.5 and 5.5]{PW07}.
\end{remark}

\excise{
\begin{remark}
  It is not yet clear to us whether the full polynomial
  $\Phi_{\!\cA}(x,y,b_1,\ldots,b_\ell)$ has an interpretation in terms
  of Poisson homology.
\end{remark}

\nicktodo{The hypertoric variety $\XA$ sits inside of the Lawrence
  toric variety $Y(\cA)$, which has dimension $\ell + \rk \cA$.  This
  torus acting on $Y(\cA)$ naturally splits into two pieces, $T^\ell$
  and $T^{\rk\cA}$.  The smaller piece $T^{\rk\cA}$ acts Hamiltonianly
  on $\XA$, which is presumably not so interesting from the standpoint
  of Poisson homology.  For the larger piece $T^\ell$, only the
  diagonal copy of $\cs$ acts on $\XA$, and this is exactly the action
  of $\cs$ that we are considering.  My hunch is that the full
  polynomial $\Phi_{\!\cA}(x,y,b_1,\ldots,b_\ell)$ can be understood
  by incorporating the rest of the torus $T^\ell$, but I don't know
  how to do this.  The problems are that $T^\ell$ doesn't preserve
  $\XA$ and (more importantly) the Poisson structure on $Y(\cA)$ is
  not homogeneous for the action of $T^\ell$.}
}

\section{The case of S3-varieties in type A}\label{sec:kostka}
Let $\la$ and $\mu$ be partitions of the same positive integer $r$.
Let $\cO_\la$ be the nilpotent coadjoint orbit in $\mathfrak{sl}_r^*$
whose Jordan blocks have sizes given by the parts of
$\la$;\footnote{More precisely, the elements of the image of this
  orbit under the Killing form isomorphism $\mathfrak{sl}_r^* \to
  \mathfrak{sl}_r$ have this Jordan decomposition.} then $\la\geq\mu$
in the dominance order if and only if $\cO_\mu$ is contained in the
closure of $\cO_\la$.  In this case, let $X_{\la\mu}$ be the normal
slice to $\cO_\mu$ inside of the closure of $\cO_\la$.  This space is
sometimes called an {\bf S3-variety}, after Slodowy, Spaltenstein, and
Springer \cite{WebWO, BLPWgco}.  The variety $X_{\la\mu}$ is a
Nakajima quiver variety for a finite type A quiver, and conversely any
such variety is an S3-variety \cite{Maf}; in particular, $X_{\la\mu}$
admits a projective symplectic resolution (the fibers are known as
{\bf Spaltenstein varieties}), and the local systems associated to the
symplectic leaves are trivial.  We equip these varieties with the
standard action of $\cs$ with the property that the Poisson bracket is
homogeneous of weight -2.  The symplectic leaves of $X_{\la\mu}$ are
indexed by the poset $[\mu,\la]$.  For any $\nu\in [\mu,\la]$, the
closure of the leaf $S_\nu$ is isomorphic to $X_{\nu\mu}$, and the
normal slice to $S_\nu$ is isomorphic to $X_{\la\nu}$.

Let $n_\la = \sum_i (i-1)\la_i$, so that 
% $\dim \cO_{\la} = r(r-1) - 2 n_\la$ and therefore
$\dim X_{\la\mu} = 2(n_\mu - n_\la)$.  A theorem of Lusztig
\cite[Theorem 2]{LuGreen} says
that \begin{equation}\label{lusztig}Q_{X_{\la\mu}}(x) = x^{2(n_\mu -
    n_\la)}K_{\la\mu}(x^{-2}),\end{equation} where $K_{\la\mu}(t)$ is
the {\bf Kostka polynomial} associated to $\la$ and $\mu$.  We will
assume that the conjecture \cite[3.4]{Pr12} holds; we have
$X_{\la \mu}^! = X_{\mu^t \lambda^{\hspace{-.5pt} t}}$, so the explicit statement of the conjecture
in this case is that
\begin{equation}\label{top}
P_{X_{\la\mu}}(0,y) = Q_{X_{\mu^t\la^{\hspace{-.5pt} t}}}(y) = y^{2(n_{\la^{\hspace{-.5pt} t}} 
  - n_{\mu^t})}K_{\mu^t\la^{\hspace{-.5pt} t}}(y^{-2}).
\end{equation}
  
\begin{proposition}\label{s3-kostka}  
If Equation \eqref{top} holds for all type A S3-varieties, then 
$$P_{X_{\la\mu}}(x,y) = y^{2(n_{\la^{\hspace{-.5pt} t}}-n_\mu)} \sum_{\nu\in [\mu,\la]} y^{2(n_\nu - n_{\nu^t})} K_{\nu\mu}(x^{2}) 
  K_{\nu^t\la^{\hspace{-.5pt} t}}(y^{-2}).$$
\end{proposition}

\begin{proof}
For all $\nu\in [\mu,\la]$, Equation \eqref{top} tells us that $$\dim \hpz(X_{\la\nu}) =
P_{X_{\la\nu}}(0,1) = K_{\nu^t\la^{\hspace{-.5pt} t}}(1),$$ which is
in turn equal to the rank of the local system $K_{S_\nu}$
\cite[3.5(b)]{BorMac}, thus the hypothesis of Theorem \ref{main} is
satisfied.  Then by Corollary \ref{pp}, we have
\begin{eqnarray*}\label{kostka-eq}\notag
  P_{X_{\la\mu}}(x,y) &=& \sum_{\nu\in [\mu,\la]} x^{2(n_\mu - n_\nu)} y^{-2(n_\mu - n_\nu)} 
  x^{2(n_\nu - n_\mu)}K_{\nu\mu}(x^2)
  y^{2(n_{\la^{\hspace{-.5pt} t}} - n_{\nu^t})}K_{\nu^t\la^{\hspace{-.5pt} t}}(y^{-2})\\
  &=& y^{2(n_{\la^{\hspace{-.5pt} t}}-n_\mu)} \sum_{\nu\in [\mu,\la]} y^{2(n_\nu - n_{\nu^t})} K_{\nu\mu}(x^{2}) 
  K_{\nu^t\la^{\hspace{-.5pt} t}}(y^{-2}). %\qedhere
\end{eqnarray*}
% We note the similarity between Equations \eqref{tutte-eq} and
% \eqref{kostka-eq}; see Remark \ref{duality} for further discussion.
This completes the proof.
\end{proof}

% It is easy to verify the cases of \eqref{top} needed for the
% nilpotent cone in types $A_2$ and $A_1$. More generally, whenever
% the slices $X_{\nu^t \la^t}$ occurring are all either of dimension
% two or are isomorphic to nilpotent orbit closures $\bar
% \cO_{\kappa}$, then the needed cases of \eqref{top} hold: in the
% first case $\hpz(\cO_{X_{\nu^t \la^t}})$ was computed in
% Alev-Lambre, and in the second case $\hpz(\cO_{X_{\nu^t \la^t}})$ is
% one-dimensional, owing to the fact that this is true for the whole
% nilpotent cone by \cite{ES10} and the zeroth Poisson homology of the
% algebra of functions on any orbit closure is a quotient of this,
% hence also one-dimensional.  (Note also that, in these cases,
% $\cO_{X_{\nu^t \la^t}}$ is actually hypertoric, so one can
% alternatively appeal to \cite[3.1]{Pr12}.)

\section{The case of the nilpotent cone in general
  type}\label{sec:general type}
Let $\mathfrak{g}$ be any semisimple Lie algebra, and let $X \subseteq
\mathfrak{g}^*$ be the nilpotent cone.  As before, one can consider
coadjoint orbits in $X$ and slices to one inside the closure of
another; however, these do not admit symplectic resolutions in
general, and even when they do, the assumptions of Theorem \ref{main}
are not known to be satisfied.  Here we consider only the case of $X$
itself, where Theorem \ref{main} is known to hold for the Springer
resolution $T^*\cB \to X$ \cite{ES10}, where $\cB$ is the flag
variety.\footnote{In fact, in \cite{ES10} the conclusion $M \cong T$ of
  Theorem \ref{main} is proved first, and then the hypothesis
  follows.} 
%  (note that the slices in $X$ to each coadjoint orbit do
%always admit symplectic resolutions, just not necessarily their
%intersections with other orbit closures). 
If $\mathfrak{g}$ is not of type A, then the hypothesis of Corollary
\ref{pp} fails, so we have no direct way of using that result to
compute $P_X(x,y)$.  However, we will conjecture a formula for
$P_X(x,y)$ based on the type A case and a suggestion of G.~Lusztig and
P.~Etingof.
%\footnote{Thanks to G.~Lusztig and P.~Etingof for suggesting a
%  formula of this type.}

\subsection{Generalized Kostka polynomials}
Let $W$ be the Weyl group of $\mathfrak{g}$.  Springer theory tells us
that $T$ is equipped with an action of $W$, and that for every
irreducible representation $\chi$ of $W$, we may associate a nilpotent
coadjoint orbit $\cO_{\mathfrak{g},\chi}$ and an irreducible local
system $M_{\mathfrak{g},\chi}$ on $\cO_{\mathfrak{g},\chi}$ such that
\begin{equation}\label{Springer}M\cong T \cong
  \bigoplus_{\mathfrak{g},\chi}
  \IC(\cO_{{\mathfrak{g},\chi}};M_{\mathfrak{g},\chi}) \otimes \chi
\end{equation} as a $W$-equivariant D-module.
% \footnote{Equivalently, for every nilpotent orbit $\cO$, we have
% $K_\cO\cong\displaystyle\sum_{\cO_\chi =
% \cO}M_\chi\otimes\chi$.} $$T\;\;\;\cong \bigoplus_{\chi \in
% \Irrep(W)}\IC(\bar\cO_\chi; M_\chi)\otimes \chi.$$
By pushing forward to a point and taking cohomology, we obtain an
action of $W$ on $\H^*(T^*\cB; \C) = H^*(\cB; \C)$ which is isomorphic
(after forgetting the grading) to the regular representation.

For each $\chi$ of $W$, let
$$K_{\mathfrak{g},\chi}(t) := \sum_{i \geq 0} t^i\dim
\Hom_W\!\big(\chi, \H^{2\dim\cB - 2i}(\cB; \C)\big).$$ We call
$K_{\mathfrak{g},\chi}(t)$ a {\bf generalized Kostka polynomial},
motivated by the following well-known proposition.

\begin{proposition}\label{typeA}
  For any $\mathfrak{g}$ and any representation $\chi$ of $W$, we have
$$K_{\mathfrak{g},\chi}(t^2) = 
\sum_{i\geq
  0}t^i\dim\IH^{\dim\cO_{\mathfrak{g},\chi}-i}(\bar\cO_{\mathfrak{g},\chi};
M_{\mathfrak{g},\chi}).$$ If
$\mathfrak{g}=\mathfrak{sl}_r$, $\chi$ is an irreducible
representation of $S_r$, and $\nu$ is the partition of $r$ with the
property that $\cO_{\mathfrak{g},\chi} = \cO_\nu$, then
$K_{\mathfrak{g},\chi}(t) = K_{\nu(1^r)}(t)$.
\end{proposition}

\begin{proof}
  The first statement follows immediately from pushing Equation
  \eqref{Springer} forward to a point and taking cohomology.  To
  obtain the second statement from the first, we use Equation
  \eqref{lusztig} (for $\mu = (1^r)$ and $\lambda = \nu$), along with
  the fact that, in type $A$, all the local systems
  $M_{\mathfrak{g},\chi}$ are trivial.
\end{proof}

\begin{remark}\label{r:p-d}
  By Poincar\'e duality, $\Hom_W(\chi, H^{2i}(\cB; \C)) \cong
  \Hom_W(\chi \otimes \sgn, H^{2\dim \cB-2i}(\cB;\C))$, thus we also
  have
\[
K_{\mathfrak{g}, \chi}(t) = \sum_{i \geq 0} t^i\dim
\Hom_W\!\big(\chi\otimes\sgn, \H^{2i}(\cB; \C)\big).
\]
\end{remark}

\begin{remark}\label{r:coinv-alg} Note that 
  $\H^{*}(\cB;\C)$ is canonically isomorphic as a $W$-equivariant
  graded algebra to the coinvariant algebra
  $\bC[\mathfrak{h}]/(\bC[\mathfrak{h}]^W_+)$, where
  $\bC[\mathfrak{h}]_+ \subset \bC[\mathfrak{h}]$ is the augmentation
  ideal and $\mathfrak{h}^*\subset \C[\mathfrak{h}]$ sits in degree 2.
\end{remark}

\subsection{The conjecture}
Since the summand
$\IC(\cO_{{\mathfrak{g},\chi}};M_{\mathfrak{g},\chi})$ of $M$ is
simple, the weak $\cs$-equivariant structure on $M$ induces a grading
on the multiplicity space $\chi$.  Let $h(\chi; t)$ be the Hilbert
series for this grading.

\begin{conjecture}\label{arbitrary type}
For each irreducible representation $\chi$ of $W$, we have  
\[
  h(\chi;y) = K_{\mathfrak{g}, \chi}(y^{-2}),
\]
and therefore $$P_{X}(x,y)\;\;\; = \sum_{\chi \in
  \Irrep(W)}
% y^{-\dim \cO_{\mathfrak{g},\chi} + \dim \cO_{\mathfrak{g}^L,i(\chi)
% \otimes \sgn}} K_{\mathfrak{g},\chi}(x^2) K_{\mathfrak{g}^L,i(\chi)
% \otimes \sgn}(y^{-2})
K_{\mathfrak{g},\chi}(x^2) K_{\mathfrak{g},\chi}(y^{-2}).$$
\end{conjecture}

\begin{remark}
  Conjecture \ref{arbitrary type} holds at the specialization $y=1$ by
the fact that $H^*(\cB; \C)$ is isomorphic to the regular representation of $W$.
\end{remark}

\begin{remark}\label{r:bottom-degree}
  If $M_{\mathfrak{g},\chi}$ is trivial, then
  $\IH^0(\bar\cO_{\mathfrak{g},\chi}; M_{\mathfrak{g},\chi}) \cong
  \C$, thus Proposition \ref{typeA} tells us that the top degree of
  $K_{\mathfrak{g},\chi}(x^2)$ is equal to the dimension of
  $\cO_{\mathfrak{g},\chi}$.  Similarly, the bottom degree of
  $K_{\mathfrak{g},\chi}(y^{-2})$ is equal to $-\dim
  \cO_{\mathfrak{g},\chi}$, which is what the bottom degree of
  $h(\chi)$ should be according to Theorem \ref{weights}.
\end{remark}

\begin{remark}
By Theorem \ref{weights}, Conjecture \ref{arbitrary type} implies
that, for each nilpotent orbit $S$,
\begin{equation}\label{e:hilb-sl}
  P_{X_S}(0,y)\;\; = \sum%_{\cO_{\mathfrak{g},\chi}=S} 
  y^{\dim\cO_{\mathfrak{g},\chi}} K_{\mathfrak{g}, \chi}(y^{-2}),
\end{equation}
where the sum is taken over all $\chi$ such that
$\cO_{\mathfrak{g},\chi}=S$.  If there is only one such $\chi$ (and
hence $M_{\mathfrak{g},\chi}$ is trivial), then Conjecture
\ref{arbitrary type} for $\chi$ is equivalent to Equation
\eqref{e:hilb-sl} for $S$.
\end{remark}

\begin{example}\label{ex:triv,sign}
  In the case where $\chi = \operatorname{triv}$, which corresponds to
  the trivial local system on the open orbit, Equation
  \eqref{e:hilb-sl} says that $h(\operatorname{triv};y)=y^{-\dim X}$.
  On the other hand, if $\chi = \sgn$, which corresponds to the cone
  point, it says $h(\sgn;y) = 1$. These conclusions both agree with
  Theorem \ref{weights}, since in both cases the Poisson homology of
  the slice is one-dimensional and concentrated in degree zero.
\end{example}

\begin{proposition}\label{type A same}
  If $\mathfrak{g} = \mathfrak{sl}_r$, then the first formula of
  Conjecture \ref{arbitrary type} agrees with Equation \eqref{top}
% (with $\la = (1^r)$ and $\mu = \nu$), 
and the second with Proposition \ref{s3-kostka}.
% (with $\la = (1^r)$ and $\mu = (r)$).
\end{proposition}

\begin{proof}
If $\cO_{\mathfrak{g},\chi} = \cO_\nu$, then Equation \eqref{top} with $\la = (r)$ and $\mu = \nu$ tells us that 
$$P_{X_{(r)\nu}}(0,y) = y^{2(n_{(1^r)} - n_{\nu^t})}K_{\nu^t(1^r)}(y^{-2}) = y^{\dim \cO_{\nu^t}}K_{\nu^t(1^r)}(y^{-2})
= y^{\dim
  \cO_{\mathfrak{g},\chi\otimes\sgn}}K_{\mathfrak{g},\chi\otimes\sgn}(y^{-2}).$$
On the other hand, the first formula of Conjecture \ref{arbitrary
  type} is equivalent to Equation \eqref{e:hilb-sl}, which says that
$$P_{X_{(r)\nu}}(0,y) = y^{\dim
  \cO_{\mathfrak{g},\chi}}K_{\mathfrak{g},\chi}(y^{-2}).$$ 
Thus we
need to prove the following identity:
\begin{equation}\label{e:id-typea}
  K_{\mathfrak{g},\chi}(t^2) = t^{\dim \cO_{\mathfrak{g},\chi} - \dim \cO_{\mathfrak{g},\chi \otimes \sgn}} K_{\mathfrak{g},\chi \otimes \sgn}(t^2). 
\end{equation}
Using Poincar\'e duality (Remark \ref{r:p-d}) and the fact that $\dim \cO_{\mathfrak{g},\chi} = r(r-1) - n_{\nu}$, this identity reduces to the following palindromic property of $K_{\mathfrak{g},\chi}(t^2)$:
\[
K_{\mathfrak{g},\chi}(t^2) = t^{n_{\nu^t}-n_{\nu} + r(r-1)} K_{\mathfrak{g},\chi}(t^{-2}).
\]
This follows from \cite[Propositions A and B, (1)]{BL-snrceWg}, and (as
explained there) is originally due to Steinberg \cite{Ste-garflgGf}.
%% This is proved by Beynon and Lusztig \cite{BL-snrceWg}; the
%% polynomial is shown to be palindromic in Section 2, and the correct
%% exponent of $t$ is computed in Section 3.
\excise{ For this, refer to Steinberg (ref.~(8) at the end of Section
  2 of \cite{BL-snrceWg}), and cf.~\cite[\S 2]{BL-snrceWg} (for the
  palindromic property without the exponent of $t$), as well as
  \cite[(1)]{BL-snrceWg} (for the exponent), which as observed there
  is essentially equivalent to a formula in the reference (6) therein.
}
\end{proof}

\vspace{-\baselineskip}
\begin{remark}\label{BL}
  As pointed out by G.~Lusztig, it is possible to generalize
  \eqref{e:id-typea} to arbitrary irreducible types.
  % The resulting formula resembles symplectic duality.
%
%  To explain this formula, we have to recall some facts about
%  generalized Kostka polynomials.
  For any irreducible representation $\chi$ of $W$, let $\chi^s$
  denote the unique special representation in the same two-sided cell
  as $\chi$.
  %% Note that $\cO_{\mathfrak{g},\chi^s} \neq
  %% \cO_{\mathfrak{g},\chi}$ in general, and many times $M_{\chi^s}$
  %% can be trivial.
  In \cite{BL-snrceWg}, there is an involution $i$ defined on the set
  of irreducible representations of $W$, which is the identity except
  for six irreducible representations in types $E_7$ and $E_8$, called
  ``exceptional'' ones, which are exactly the representations for
  which $K_{\mathfrak{g},\chi}(t)$ is not palindromic.

  Lusztig pointed out that, combining \cite[Propositions A and
  B]{BL-snrceWg} with the determinant of
  \cite[5.12.2]{Lusztig-crgff-book}, and comparing powers of $u$ in
  the latter, one can conclude the following identity (when $W$ is
  irreducible):
\begin{equation}\label{e:BL}
  K_{\mathfrak{g},\chi}(t^2) = t^{\dim
    \cO_{\mathfrak{g},\chi^s} - \dim \cO_{\mathfrak{g},\chi^s \otimes
      \sgn}} K_{\mathfrak{g},i(\chi) \otimes \sgn}(t^2).
    \end{equation} 
    In type $A$, $i$ is trivial and $\chi=\chi^s$ for all $\chi$, thus
    we recover the identity in Equation \eqref{e:id-typea}.
\end{remark}

\begin{remark}\label{r:old-conj}
Motivated by in part by symplectic duality \cite{BLPWgco}, 
% which is supposed to involve replacing $\mathfrak{g}$ with its
% Langlands dual $\mathfrak{g}^L$ (and often involves replacing the
% slice to $\cO_{\mathfrak{g},\chi}$ with the orbit closure $\bar
% \cO_{\mathfrak{g}^L, \chi \otimes \sgn}$),
we originally guessed the following formula for $h(\chi)$:
\[
y^{-\dim \cO_{\mathfrak{g},\chi} + \dim \cO_{\!\mathfrak{g}^L\!,\chi \otimes \sgn}}
K_{\mathfrak{g},\chi \otimes \sgn}(y^{-2}).
\]
(Here $\mathfrak{g}^L$ is the Langlands dual of $\mathfrak{g}$, whose
Weyl group is canonically isomorphic to that of $\mathfrak{g}$.)  This
agrees with Conjecture \ref{arbitrary type} in all of the examples
considered in this paper: types $A_\ell$, $B_2$, $C_2$, and $G_2$, and
also for the subregular orbit in general (and, in the $B_2, C_2$, and
$G_2$ cases, the Langlands duality is required for it to hold).
However, as Lusztig pointed out, the formulas do not coincide in some
cases, such as when $\chi$ is the (non-exceptional) $50$-dimensional
irreducible representation of $E_8$ for which $M_{\mathfrak{g},\chi}$
is trivial; moreover, Remark \ref{r:bottom-degree} implies that our
original guess was incorrect in this case.
% (note also that $i(\chi)=\chi$, i.e., $K_{\mathfrak{g},\chi}$ is
% palindromic).
\end{remark}

\subsection{A proof of the conjecture along the subregular orbit}
In this subsection we verify Conjecture \ref{arbitrary type} when
$\cO_{\mathfrak{g},\chi}$ is equal to the subregular orbit $R$.
% , which always includes the reflection representation.
First suppose $\mathfrak{g}$ is simply laced; in this case, the only
such representation is the reflection representation
$\chi=\mathfrak{h}$.

\begin{proposition}\label{ex:sr-sl}
If $\mathfrak{g}$ is simply laced, then Conjecture \ref{arbitrary type} holds for $\chi = \mathfrak{h}$.
\end{proposition}

\begin{proof}
  Since there is only one irreducible representation associated to the
  subregular orbit, Conjecture \ref{arbitrary type} for $\mathfrak{h}$
  is equivalent to Equation \eqref{e:hilb-sl} for $R$, which says
$$P_{X_R}(0,y) = y^{\dim R}K_{\mathfrak{g},\mathfrak{h}}(y^{-2}).$$
% Since $\mathfrak{h}$ is both special and fixed by the involution
% $i$, Equation \eqref{e:BL} tells us that
% \[
% y^{\dim R}K_{\mathfrak{g},\mathfrak{h}}(y^{-2}) = y^{\dim
% \cO_{\mathfrak{g},\mathfrak{h} \otimes \sgn}} K_{\mathfrak{g},
% \mathfrak{h} \otimes \sgn}(y^{-2}).
% \]
By Remarks \ref{r:p-d} and \ref{r:coinv-alg}, since $\codim R = 2$, we have
\[
y^{\dim R}K_{\mathfrak{g},\mathfrak{h}}(y^{-2}) = 
y^{-2}h(\Hom_W(\mathfrak{h}, \bC[\mathfrak{h}]/(\bC[\mathfrak{h}]^W_+);y).
\]
Consider the map from $\Phi:\C[\mathfrak{h}]^W\to\Hom_W(\mathfrak{h},
\bC[\mathfrak{h}] / (\bC[\mathfrak{h}]^W_+))$ taking $f$ to $\Phi_f$,
which is defined by the formula $\Phi_f(x) := \partial_x(f)$ for all
$x \in \mathfrak{h}$.  The restriction of $\Phi$ to the linear span of
the fundamental invariants (the ring generators of
$\bC[\mathfrak{h}]^W$) is an isomorphism.  Since $\Phi$ lowers degree
by 2, this implies that
$$y^{\dim R}
K_{\mathfrak{g},\mathfrak{h}}(y^{-2}) = y^{-2} \sum_i y^{2d_i-2} =
\sum_i y^{2(d_i-2)},$$ where $\{2d_i\}$ are the degrees of the
fundamental invariants.\footnote{The factor of 2 is there because
  $\mh^*$ sits in degree 2.}  This indeed coincides with
$P_{X_R}(0,y)$, as desired \cite{Gre-GMZ, AL98}.
\end{proof}

In the non-simply laced case, let $\tilde D$ be the simply laced
Dynkin diagram folding to the type of $\mathfrak{g}$, let $\tilde W$
be the corresponding Weyl group, and let $\tilde{\mathfrak{h}}$ be its
reflection representation. As representations of $W$, we have
$\tilde{\mathfrak{h}} \cong \mathfrak{h} \oplus \tau$ for some
irreducible representation $\tau\not\cong\mathfrak{h}$ of $W$, and
$\tau$ and $\mathfrak{h}$ are the only two irreducible representations
lying over $R$.
%% The representation $\mathfrak{h}$ is special, and we have $\tau^s =
%% \mathfrak{h}$ and $(\tau\otimes\sgn)^s = \mathfrak{h}\otimes\sgn$.
The slice $X_R$ is a Kleinian singularity of type $\tilde D$, and
Theorem \ref{weights} tells us that $\hpz(X_R)$ is isomorphic as a
graded vector space to a fiber of the local system $L_R[-\dim R]$,
where
$$L_R = \left(M_{\mathfrak{g},\mathfrak{h}}\otimes\mh \right)\oplus\left(M_{\mathfrak{g},\mathfrak{\tau}}\otimes\mathfrak{\tau}\right).$$

\begin{proposition}\label{ex:sr-nsl}
  If $\mathfrak{g}$ is not simply laced, then Conjecture
  \ref{arbitrary type} holds for $\chi = \mathfrak{h}$ and for $\chi =
  \tau$.
\end{proposition}

\begin{proof}
  \excise{ To prove the conjecture for $\chi=\mathfrak{h}$, first note
    that this is special. Next, $\cO_{\mathfrak{g}^L,\mathfrak{h}
      \otimes \sgn}$ is isomorphic to the quotient of the minimal
    nilpotent orbit in type $\tilde D$ by the group of automorphisms
    of the graph $\tilde D$ ($S_2$ when $\mathfrak{g}$ is of types
    $B_n, C_n$, or $F_4$, in which case $\tilde D$ is $A_{2n-1},
    D_{n+1}$, or $E_6$, respectively; and $S_3$ when it is of type
    $G_2$, in which case $\tilde D = D_4$). The dimension, as in
    Example \ref{ex:refl-sl}, is $2h-2$, where $h$ is the Coxeter
    number of $\tilde D$.  } For $\chi = \mh$ or $\tau$, let
  $\hpz(X_R)_\chi\subset \hpz(X_R)$ be the summand corresponding to a
  fiber of the local system $M_{\mathfrak{g},\mathfrak{h}}[-\dim
  R]\subset L_R[-\dim R]$.  As in the proof of Proposition
  \ref{ex:sr-sl}, we need to show that the Hilbert series of
  $\hpz(X_R)_\chi$ is equal to
\begin{equation}\label{e:chi}
y^{-2}
h(\Hom_W(\chi, \bC[\mathfrak{h}] /
(\bC[\mathfrak{h}]^W_+));y).
\end{equation}

We first consider the case where $\chi = \mh$.  The local system
$M_{\mg,\mh}$ is trivial, so $\hpz(X_R)_\mh$ is the part of
$\hpz(X_R)$ that is fixed by the action of $\pi_1(R)$.  As in the
proof of Proposition \ref{ex:sr-sl}, Equation \eqref{e:chi} simplifies
to $$\sum_i y^{2(d_i-2)},$$ where $\{2d_i\}$ are the degrees of the
fundamental invariants for the action of $W$ on $\bC[\mh]$.  (Note
that these are a subset of the fundamental invariants for the action
of $\tilde W$ on $\bC[\tilde\mh]$.)  We will check on a case-by-case
basis that this is equal to the Hilbert series of $\hpz(X_R)_\mh$.  We
will skip the case of $G_2$, since that will be treated separately in
Proposition \ref{ex:g2}.  In all other cases, $\pi_1(R)\cong \bZ/2$,
and the action on $\hpz(X_R)$ can be deduced from the explicit bases
for the latter in \cite[\S 5.1]{EGPRS}.  It is straightforward to
check that our formula is correct.

\excise{
The latter degrees are as
follows: for $\mathfrak{g}$ of type $B_n$, $d_i \in
\{3,5,\ldots,2n-1\}$; for type $C_n$, they are $d_i \in \{n\}$, for
$\mathfrak{g}$ of type $F_4$, $d_i \in \{5,9\}$; and for
$\mathfrak{g}$ of type $G_2$, $d_i = 4$ with multiplicity two.  We
will assume we are not in the case $G_2$, which is handled in Proposition
\ref{ex:g2} below.  Otherwise, $\pi_1(S) = \bZ/2$ and the monodromy
around the nontrivial loop is induced by the nontrivial symplectic
automorphism of $X_S$. One can check explicitly (using a basis for
$\hpz(X_S)$, which in the $B_n$ and $C_n$ cases are given in \cite[\S
5.1]{EGPRS}, where $X_S$ is Kleinian of types $A_{2n-1}$ and
$D_{n+1}$, respectively) that the degrees in which $\pi_1(S)$ acts
nontrivially are the ones listed above.
}

Next, consider the case $\chi = \tau$.  In view of the above, we need
to show that
$$y^{-2}h(\Hom_W(\tau, \bC[\mathfrak{h}] / (\bC[\mathfrak{h}]^W_+));y) = \sum_i y^{2(e_i - 2)},$$
where $\{2e_i\}$ are the degrees of the fundamental invariants for the
action of $\tilde W$ on $\bC[\tilde\mh]$ that restrict to zero on $\mh
\subseteq \tilde \mh$.  
% \travistodo{Again, the $\tilde W$ (as well as
%  the restriction) is crucial (and here the statement is false
%  otherwise).  We don't want to use an inclusion $\bC[\mh] \subseteq
%  \bC[\tilde \mh]$.}  
To prove this, it is sufficient to show that
there exists a graded vector space isomorphism
% Recall now the decomposition $\tilde{\mathfrak{h}} =
% \mathfrak{h} \oplus \chi$ (as representations of $W$) above.
% In view of our result for
% $h(\Hom_W(\mathfrak{h}, \bC[\mathfrak{h}] / (\bC[\mathfrak{h}]^W_+));y)$,
% and similarly replacing $W$ and $\mathfrak{h}$ by $\tilde W$ and $\tilde{\mathfrak{h}}$, it suffices to prove that
\[
\Hom_{\tilde W}(\tilde{\mathfrak{h}},
\bC[\tilde{\mathfrak{h}}]/(\bC[\tilde{\mathfrak{h}}]^{\tilde W}_+)) \cong
\Hom_W(\tilde{\mathfrak{h}},
\bC[\mathfrak{h}]/(\bC[\mathfrak{h}]^W_+)).
\]
The restriction map from $\bC[\tilde{\mathfrak{h}}]$ to $\bC[\mathfrak{h}]$ induces a natural map from the left-hand side
to the right-hand side.  
Moreover, both sides have the
same dimension (equal to $\dim \tilde{\mathfrak{h}})$, since the
coinvariant algebras for $W$ and $\tilde W$ are the regular
representations of $W$ and $\tilde W$, respectively.  Therefore, it
suffices to prove that the natural map is injective.

Equivalently, we need to show that, for every fundamental invariant $f
\in \bC[\tilde {\mathfrak{h}}]^{\tilde W}$ which restricts to zero on
$\tilde \mh$, the corresponding homomorphism $\Phi_f \in \Hom_{\tilde
  W}(\tilde{\mathfrak{h}}, \bC[\tilde{\mathfrak{h}}] /
(\bC[\tilde{\mathfrak{h}}]^{\tilde W}_+))$ defined above restricts to
a nonzero element of $\Hom_{W}(\tau,
\bC[\mathfrak{h}]/(\bC[\mathfrak{h}]^W_+))$.  This is easy to verify
explicitly in the case where $\mathfrak{g}$ is of type $B_n$ (so
$\tilde D = A_{2n-1}$), using the embedding $W(B_n) \into W(A_{2n-1}) \cong
S_{2n-1}$, since then $\bC[\tilde{\mathfrak{h}}]^{\tilde W}$ is the
ring of symmetric polynomials (modulo linear symmetric polynomials).
In the case $C_n$, $\tau$ is one-dimensional and $\tau \otimes \tau$
is trivial, thus $\mh^\perp \subseteq (\tilde \mh)^*$ is
one-dimensional.  Then, the fundamental invariant $f$ of
$\bC[\tilde{\mathfrak{h}}]^{\tilde W}$ which restricts to zero in
$\bC[\mathfrak{h}]^W$ lies in $(\mh^\perp)$ but not in
$(\mh^\perp)^2$.  It follows that the corresponding element $\Phi_f$
indeed restricts to a nonzero element of $\Hom_W(\tau,
\bC[\mh]/(\bC[\mh]^W_+))$.  In the case $F_4$, one can explicitly
verify the statement.
% (as in the case $G_2$, although the conjecture in
% that case also follows from Proposition \ref{ex:g2} below).
\end{proof}

\subsection{Proof of the conjecture for semisimple Lie algebras of rank at most 2}
Conjecture \ref{arbitrary type} is easy to verify for $\mathfrak{g}$
of type $A_1$ and $A_2$ by checking Equation \eqref{top} in low
dimensions.\footnote{In fact, in these cases, the result also follows
  from Example \ref{ex:triv,sign} and Proposition \ref{ex:sr-sl}.}
% In rank 1 or 2, the variety $X_{\la\mu}$ is always either a
% nilpotent orbit closure, in which case \cite{ES10} implies that
% $P_{X_{\la\mu}}(0,y)=1$, or a Kleinian singularity, in which case
% $\hpz$ and its grading are known (and the statement also follows
% from Proposition \ref{ex:sr-sl}).
In the two remaining
examples, we prove the conjecture for $\mathfrak{g}$ of type $B_2$ and
$G_2$, and therefore for all all $\mathfrak{g}$ of semisimple rank at
most 2.

\begin{proposition}\label{ex:b2-c2}
Conjecture \ref{arbitrary type} for $\mathfrak{g}$ of type $B_2$ ($\mg = \mathfrak{so}_5$).
\end{proposition}

\begin{proof}
%  The nilpotent cone $X \subseteq \mathfrak{g}^*$ has
%  dimension eight.  
There are four nilpotent orbits: the zero
  orbit, the minimal orbit (of dimension four), the subregular orbit
  (of dimension six), and the open orbit (of dimension eight).  Call these $O_0, O_2, O_4,
  O_6$, and $O_8$, where $O_k$ has dimension $k$.  These orbits are
  all simply-connected except for $O_6$, which has fundamental group
  $\bZ/2\bZ$.  Let $\Omega_k$ denote the rank-one trivial local system
  on $O_k$, and let $L_6$ be the nontrivial rank-one local system
  (with regular singularities) on $O_6$.  The Weyl group is 
  % the wreath
  % product $(\bZ/2\bZ)^2 \rtimes S_2$ (equivalently, the dihedral
  % group
  % of order eight),
  isomorphic to the dihedral group of order eight, which has five
  irreducible representations: $\triv, \sgn, \tau, \tau \otimes \sgn$,
  and $\mathfrak{h}$.
  % The Langlands dual group $\mathfrak{g}^L$ is of type $C_2$ (that
  % is, $\mathfrak{g}^L=\mathfrak{sp}_4$).  In fact, there is an
  % isomorphism $\mathfrak{g} \cong \mathfrak{g}^L$, thus we may
  % identify the $\mathfrak{g}^L$ nilpotent orbits with the
  % $\mathfrak{g}$ nilpotent orbits.  However, the automorphism of $W$
  % induced by this isomorphism of Lie algebras is not the identity,
  % so it is not necessarily the case that $\cO_{\mathfrak{g},\chi} =
  % \cO_{\mathfrak{g}^L,\chi}$ for every irreducible representation
  % $\chi$ of $W$. Indeed,
The Springer correspondence for $\mg$ takes the following form
  \cite[\S 13.3]{Carter}.\\
  
\begin{center}
  \begin{tabular}{|c|c|} \hline $\chi$ & $(\cO_{\mathfrak{g},\chi},

    M_{\mathfrak{g},\chi})$ % & $(\cO_{\mathfrak{g}^L, \chi}, M_{\mathfrak{g}^L, \chi})$
    \\ \hline $\triv$ & $(O_8, \Omega_8)$ % & $(O_8, \Omega_8)$
    \\ \hline $\sgn$ & $(O_0, \Omega_0)$ % & $(O_0, \Omega_0)$
    \\ \hline $\tau$ & $(O_6, L_6)$ % & $(O_4, \Omega_4)$
    \\ \hline $\tau \otimes \sgn$ & $(O_4, \Omega_4)$ % & $(O_6, L_6)$
    \\ \hline $\mathfrak{h}$ & $(O_6, \Omega_6)$ % & $(O_6, \Omega_6)$
    \\ \hline
  \end{tabular}\end{center}
  \vspace{\baselineskip}
  %% We have $\tau^s = (\tau\otimes\sgn)^2 = \mathfrak{h}$, and all
  %% other irreps are special. Note that the superscript of L in
  %% Remark \ref{r:old-conj} matters since $\dim
  %% \cO_{\mathfrak{g}^L,\chi}
  %% \neq \dim \cO_{\mathfrak{g},\chi}$ for $\chi = \tau$ or $\tau
  %% \otimes \sgn$.
The $W$-equivariant Poincar\'e
polynomial of the coinvariant algebra is equal to
\[
1 + \mathfrak{h} \cdot t^2 + (\tau + \tau \otimes \sgn) \cdot t^4 + \mathfrak{h} \cdot t^6 + \sgn \cdot t^8,
\]
therefore 
\[
K_{\mathfrak{g},\triv}(t^2)=t^8, \quad K_{\mathfrak{g},\sgn}(t^2)=1, \quad
K_{\mathfrak{g},\tau}(t^2) = t^4 = K_{\mathfrak{g},\tau \otimes \sgn}(t^2),
\quad K_{\mathfrak{g},\mathfrak{h}}(t^2) = t^2+t^6.
\]
Thus, Conjecture \ref{arbitrary type} says that
% (putting the summands in the order $\triv, \sgn, \tau, \tau\otimes
% \sgn$, and $\mathfrak{h}$):
$$h(\triv;y) = y^{-8}, \quad 
h(\sgn;y) = 1,\quad
h(\tau;y) = y^{-4} =h(\tau\otimes\sgn;y),\quad
h(\mathfrak{h};y)=y^{-2}+y^{-6}.$$ 
% and therefore
% \[
% P_X(x,y) = y^{-8}x^8 + 1 + x^4 y^{-4} + x^4 y^{-4} +
% (x^2+x^6)(y^{-2}+y^{-6}).
% \]
All of the slices except the slice to $O_6$ have one-dimensional
$\hpz$, therefore the conjectural formulas for $\triv$, $\sgn$, and
$\tau\otimes\sgn$ follow from Theorem \ref{weights}.  Our table tells
us that $\IC(O_6; \Omega_6)$ appears in $M$ with multiplicity $2 =
\dim\mh$ and $\IC(O_6; \Omega_6)$ appears in $M$ with multiplicity $1
= \dim\sgn$.  The slice to $O_6$ is a Kleinian singularity of type
$A_2$, where a basis for $\hpz$ is given by the images of $1, xy,
(xy)^2 \in \bC[x,y]^{\bZ/3}$.  Since the generator in top degree can
be taken to be the square of the generator in middle degree, we see
that the nontrivial local system $L_6$ must be in middle degree and
the trivial one $\Omega_6$ must be in top and bottom degrees; this
allows us to conclude that the formulas for $h(\sgn;y)$ and $h(\mh;y)$
are correct.
\end{proof}

\vspace{-\baselineskip}
\begin{proposition}\label{ex:g2}
Conjecture \ref{arbitrary type} holds for $\mathfrak{g}$ of type $G_2$.
\end{proposition}

\begin{proof}
  There are five nilpotent orbits, call them $O_0, O_{6}, O_{8},
  O_{10}$, and $O_{12}$ (again $\dim O_k = k$), and these are all
  simply-connected except for the subregular orbit $O_{10}$, which has
  fundamental group $S_3$ \cite[p.~427]{Carter}. Let $\Omega_k$ denote
  the trivial local system on $O_k$, and on $O_{10}$, let $L_{10}$
  denote the local system corresponding to the reflection
  representation of the fundamental group $S_3$ (this is irreducible
  of rank two, with regular singularities).

  % Note that, just like in the $B_2$ case, $\mathfrak{g} \cong
  % \mathfrak{g}^L$, but the Lie algebra isomorphism \emph{does not}
  % yield the same isomorphism of Weyl groups as the one coming from
  % the Coxeter presentation.  Indeed, the composition of one
  % isomorphism with the inverse of the other produces the outer
  % automorphism which swaps the two nontrivial one-dimensional
  % irreducible representations other than the sign representation
  % (and leaves the other representations fixed).
Let $\tau$ be the irreducible one-dimensional representation of $W$
other than $\sgn$ (it is denoted by $\phi_{1,3}'$ in \cite[p.~412]{Carter}).
Then the Springer
correspondence for $\mg$ takes the following form \cite[p.~427]{Carter}.\\

 \begin{center}
 \begin{tabular}{|c|c|} 
   \hline $\chi$ & $(\cO_{\mathfrak{g},\chi},
   M_{\mathfrak{g},\chi})$ % & $(\cO_{\mathfrak{g}^L, \chi}, M_{\mathfrak{g}^L, \chi})$ 
   \\ \hline $1$ & $(O_{12}, \Omega_{12})$ % & $(O_8, \Omega_8)$ 
   \\ \hline $\sgn$ & $(O_0, \Omega_0)$ % & $(O_0, \Omega_0)$ 
   \\ \hline $\tau$ & $(O_{10}, L_{10})$ % & $(O_6, \Omega_6)$ 
   \\ \hline $\tau \otimes \sgn$ & $(O_6, \Omega_6)$ % & $(O_{10}, L_{10})$ 
   \\ \hline $\mathfrak{h}$ & $(O_{10}, \Omega_{10})$ % & $(O_{10}, \Omega_{10})$ 
   \\ \hline $\mathfrak{h} \otimes \tau$ & $(O_8, \Omega_8)$ % & $(O_8, \Omega_8)$
   \\ \hline
  \end{tabular}
   \end{center}
  \vspace{\baselineskip}
  % Note also that $\chi^s = \mathfrak{h}$ for all $\chi$ other than
  % $1$ and $\sgn$.
  % As in the $B_2$ and $C_2$ cases, note that
  % $\dim \cO_{\mathfrak{g},\tau}
  % \neq \dim \cO_{\mathfrak{g}^L,\tau}$,
  % so the superscript of $L$ in Remark \ref{r:old-conj}
  % matters.

The $W$-equivariant Poincar\'e polynomial of the coinvariant algebra is equal to
\[
1 + \mathfrak{h} \cdot t^2 + (\mathfrak{h} \otimes \tau) \cdot t^4 + (\tau + \tau \otimes \sgn) \cdot t^6 + (\mathfrak{h} \otimes \tau) \cdot t^8
+ \mathfrak{h} \cdot t^{10} + \sgn \cdot t^{12}
\]
% Note that this follows immediately from the nonequivariant
% polynomial using the Poincar\'e duality $V \cdot t^m \leftrightarrow
% V^* \otimes \sgn \cdot t^{2\dim \cB-m}$, with $V \cong V^*$ as for
% all Weyl groups, and the fact that if we set $t=1$ we get the
% regular representation of $W$.
therefore
\[
K_{\mathfrak{g},1}(t^2)=t^{12}, \quad K_{\mathfrak{g},\sgn}(t^2)=1, \quad
K_{\mathfrak{g},\tau} = t^6 = K_{\mathfrak{g},\tau \otimes \sgn},
\]\[ K_{\mathfrak{g},\mathfrak{h}}(t^2) = t^2+t^{10}, \quad K_{\mathfrak{g},\mathfrak{h} \otimes \tau}(t^2)=t^4+t^8.
\]
Thus, Conjecture \ref{arbitrary type} says 
% (putting the summands in the order $1, \sgn, \tau, \tau\otimes \sgn,
% \mathfrak{h}$, and $\mathfrak{h} \otimes \tau$)
that
\[
h(\triv;y) = y^{-12}, \quad
h(\sgn;y) = 1,\quad
h(\tau;y) = y^{-6} = h(\tau\otimes \sgn;y),\]\[
h(\mh;y) = y^{-2}+y^{-10},\quad
h(\mh\otimes\tau;y) = y^{-4}+y^{-8}.
% P_X(x,y) = y^{-12}x^{12} + 1 + x^6 y^{-6} + x^6 y^{-6} +
% (x^2+x^{10})(y^{-2}+y^{-10}) + (x^4+x^8)(y^{-4}+y^{-8}).
\]
The slices to $O_0$, $O_6$, and $O_{12}$ have one-dimensional $\hpz$,
therefore the conjectural formulas for $\triv$, $\sgn$, and $\tau
\otimes \sgn$ follow from Theorem \ref{weights}.  The slice to
$O_{10}$ is a Kleinian singularity of type $D_4$, thus
$\hpz(X_{O_{10}})$, has the Hilbert series $1+2t^4+t^8$.  Since
$L_{10}$ has rank two, it must occur in weight $4$; this proves our
conjecture for $\tau$ and $\mathfrak{h}$.

Finally, to prove our conjecture for $\mathfrak{h} \otimes \tau$, we
need to show that $h(\hpz(X_{O_8});t)=1+t^4$.  First note that the
dimension of $\hpz(X_{O_8})$ must be two, as a consequence of
\cite{ES10}.  Since $X_{O_8}$ is conical and singular, the function $1
\in \C[X_{O_8}]$ has nonzero image in $\hpz(X_{O_8})$, thus we only
need to show that there is a nonzero element of $\hpz(X_{O_8})$ in
degree four.  To do this we can use the explicit realization of
$X_{O_8}$ given in \cite{GG}: it is the intersection of the nilpotent
cone with the Slodowy slice $Y := \Phi(e + \ker(\operatorname{Ad}
f))$, with $\Phi: \mathfrak{g} \to \mathfrak{g}^*$ given by the
Killing form, with $e \in O_8$ and $(e,h,f)$ a corresponding
$\mathfrak{sl}_2$-triple.  Since there is only one nilpotent orbit of
dimension 8, it is easy to see that we can take $e$ to be the
generator $e_\alpha$ of the root space for the short simple root
$\alpha$, $f = f_\alpha$, and $h=h_\alpha$.  Moreover, as explained in
\cite{ES10}, it suffices to compute $\hpz(Y)$ itself, since this is a
free module over $\bC[\mathfrak{g}]^{\mathfrak g}$, with
$\hpz(X_{O_8}) \cong \hpz(Y) / (\bC[\mathfrak{g}]^{\mathfrak g}_+)$,
the latter being the augmentation ideal.  The latter can be computed
explicitly in the first few degrees: under the Kazhdan grading,
$\C[Y]$ is a polynomial algebra on generators in degrees $2,2,2,4$,
$5$, and $5$. The first three generators are the $\mathfrak{sl}_2$
triple mentioned above, and they act trivially on the generator in
degree $4$.  Thus in degree four, $\hpz(Y)$ has dimension two.
However, the generators of $\bC[\mathfrak{g}]^{\mathfrak{g}}$ are in
degrees four and twelve (these are the fundamental invariants, and the
Kazhdan grading restricts on $\bC[\mathfrak{g}]^{\mathfrak{g}}$ to the
the usual grading placing $\mathfrak{g}^*$ in degree two, and the
latter is well-known to assign the generators degrees four and
twelve).  Thus, in degree four, $\hpz(X_{O_8})$ has dimension one, as
desired.
\end{proof}

\bibliography{./symplectic}
\bibliographystyle{amsalpha}
\end{document}